\documentclass{amsart}[12pt]
\def\doctype{}

\usepackage{latexsym,amssymb,bm}
\usepackage{color}
\usepackage{dsfont}
\usepackage{fancyhdr}
\usepackage{lscape}
\usepackage{tikz}
\usepackage{hyperref}

\newcommand\one{\mathds{1}}

\newcommand\Q{\mathbb{Q}}
\newcommand\R{\mathbb{R}}

\newcommand{\comment}[1]{}

\numberwithin{equation}{section}

\fancypagestyle{titlepage}{
\fancyhead[R]{\doctype}
\fancyhead[CL]{}
\cfoot{\vspace{5pt} \thepage}
}

\let\oldsection\section
\newcommand\boldsection[1]{\oldsection{\bf #1}}
\newcommand\starsection[1]{\oldsection*{\bf #1}}
\makeatletter
\renewcommand\section{\@ifstar\starsection\boldsection}
\makeatother

\newtheoremstyle{theorem}
  {12pt}		  
  {0pt}  
  {\sl}  
  {\parindent}     
  {\bf}  
  {. }    
  { }    
  {}     
\theoremstyle{theorem}
\newtheorem{thm}{Theorem}[section]  
\newtheorem{lemma}[thm]{Lemma}     

\newtheorem{conj}[thm]{Conjecture}
\newtheorem{prop}[thm]{Proposition}

\newtheoremstyle{definition}
  {12pt}		  
  {0pt}  
  {}  
  {\parindent}     
  {\bf}  
  {. }    
  { }    
  {}     
\theoremstyle{definition}

\newtheorem{ex}[thm]{Example}
\newtheorem{cons}[thm]{Construction}

\newcommand\rks{{\sc Remarks.} }

\renewcommand{\proofname}{Proof}

\makeatletter
\renewenvironment{proof}[1][\proofname]{\par
  \pushQED{\qed}%
  \normalfont \partopsep=\z@skip \topsep=\z@skip
  \trivlist
  \item[\hskip\labelsep
        \scshape
    #1\@addpunct{.}]\ignorespaces
}{%
  \popQED\endtrivlist\@endpefalse
}
\makeatother

\makeatletter
\renewcommand*\@maketitle{%
  \normalfont\normalsize
  \@adminfootnotes
  \@mkboth{\@nx\shortauthors}{\@nx\shorttitle}%
  \global\topskip42\p@\relax 
  \@settitle
  \ifx\@empty\authors \else {\vskip 1em
\vtop{\centering\shortauthors\@@par}} \fi
  \ifx\@empty\@date \else {\vskip 1em \vtop{\centering\@date\@@par}}\fi 
  \ifx\@empty\@dedicatory
  \else
    \baselineskip18\p@
    \vtop{\centering{\footnotesize\itshape\@dedicatory\@@par}%
      \global\dimen@i\prevdepth}\prevdepth\dimen@i
  \fi
  \@setabstract
  \normalsize
  \if@titlepage
    \newpage
  \else
    \dimen@34\p@ \advance\dimen@-\baselineskip
    \vskip\dimen@\relax
  \fi
} 
\renewcommand*\@adminfootnotes{%
  \let\@makefnmark\relax  \let\@thefnmark\relax
  \ifx\@empty\@subjclass\else \@footnotetext{\@setsubjclass}\fi
  \ifx\@empty\@keywords\else \@footnotetext{\@setkeywords}\fi
  \ifx\@empty\thankses\else \@footnotetext{%
    \def\par{\let\par\@par}\@setthanks}%
  \fi
\thispagestyle{titlepage}
}
\makeatother

\begin{document}

\title[Cone generated by triangles]{\large On the cone of weighted graphs\\ generated by triangles}

\author{Coen del Valle}
\address{\rm Coen del Valle: Mathematics and Statistics,
University of Victoria, Victoria, BC, Canada}
\email{cdelvalle@uvic.ca}

\author{Peter J.~Dukes}
\address{\rm Peter J. Dukes:
Mathematics and Statistics,
University of Victoria, Victoria, BC, Canada
}
\email{dukes@uvic.ca}

\author{Kseniya Garaschuk}
\address{\rm Kseniya Garaschuk:
Mathematics and Statistics,
University of the Fraser Valley, Abbotsford, BC, Canada
}
\email{kseniya.garaschuk@ufv.ca}

\date{\today}

\begin{abstract}
Motivated by problems involving triangle-decompositions of graphs, we examine the facet structure of the cone $\tau_n$ of weighted graphs on $n$ vertices generated by triangles.  Our results include enumeration of facets for small $n$, a construction producing facets of $\tau_{n+1}$ from facets of $\tau_n$, and an arithmetic condition on entries of the normal vectors.  We also point out that a copy of $\tau_n$ essentially appears via the perimeter inequalities at one vertex of the metric polytope.
\end{abstract}

\thanks{This research is supported by NSERC grant 312595--2017}

\maketitle
\hrule

\bigskip

\section{Introduction}
\label{sec-intro}

\subsection{Overview}

The problem of partitioning the edges of a graph $G$ into triangles has been studied extensively.   When $G$ is a complete graph, the problem is equivalent to the existence of a Steiner triple system, an object dating back at least to the mid-19th century, \cite{Kirkman,Steiner}.  For general graphs, the decision problem for existence of a triangle decomposition is known \cite{DT} to be NP-complete.  In view of this, attention has turned to sufficient conditions for special graph families.  Planar graphs are considered in \cite{MvB}. The problem for certain $3$-partite graphs $G$ has been studied \cite{BKLOT,BD} for its connection to partial latin square completion.  Otherwise, most research on triangle decompositions has  focused on host graphs $G$ having high minimum degree.  A sequence of papers including \cite{Gus}, then \cite{BKLO,Dross} and more recently \cite{DP} have lowered the sufficient minimum degree threshold toward the conjectured limit of Nash-Williams, \cite{NW}.

Of course, for a graph (or even a multigraph) $G$ to admit a partition of its edges into triangles, it is necessary for the number of edges of $G$ to be a multiple of three, and for each degree of $G$ to be even.  These are `arithmetic' necessary conditions, since they arise from the integrality of triangle weights used.  Recent work on the problem often considers the fractional relaxation, in which triangle weights can be nonnegative reals.  Indeed, concerning the minimum degree threshold, the breakthrough result \cite{BKLO} essentially reduces the problem to its fractional relaxation.  

In the fractional relaxation, arithmetic conditions disappear, but there remain other necessary conditions.  
Suppose we map the vertices of $G$ into the unit interval.  The perimeter of every triangle of $G$ in this embedding is at most $2$.  It follows that for $G$ to have a triangle decomposition, the average length of an edge cannot exceed $2/3$.  But there exist graphs on $n$ vertices with minimum degree approaching $3n/4$ from below which fail this condition; for instance, a blow-up $C_4 \cdot K_{m}$ of the four-cycle by equal-sized cliques has, when vertices are mapped to $\{0,1\}$ according to a $2$-colouring of the underlying $C_4$, has average edge-length
$$\frac{(2m)^2}{(2m)^2+4\binom{m}{2}} > \frac{2}{3}.$$
This example was given (though analyzed somewhat differently) by Ron Graham at the end of Nash-Williams' note \cite{NW}.  Combining this with the arithmetic conditions, the conjecture attributed to Nash-Williams normally reads as follows.

\begin{conj}[\cite{NW}]
\label{nw-conj}
For sufficiently large $n$, every graph $G$ on $n$ vertices with even degrees, number of edges a multiple of three, and minimum degree at least $3n/4$ has an edge-decomposition into triangles.
\end{conj}

It is natural to call the condition on average edge length a `geometric' necessary condition since it is witnessed by a Euclidean embedding.  But there is an even more general geometric viewpoint, namely that a graph built as a nonnegative combination of triangles inherits `by convexity' any linear constraint on its triangles.

The purpose of this paper is to take some preliminary steps toward organizing geometric necessary conditions for triangle decompositions.  We introduce and study a polyhedral cone relevant for the problem (although the cone has essentially appeared before in some other contexts).  The description of this cone by facets seems very complicated, but offers an encoding of all constraints for (the fractional relaxation of) the triangle decomposition problem.  We believe an understanding of this cone is possibly useful for making further steps toward Conjecture~\ref{nw-conj}.

\subsection{Set-up and notation}

For our purposes, a \emph{weighted graph} on a vertex set $V$ is a function $f:\binom{V}{2} \rightarrow \R$ assigning a real number to each edge of the complete graph on $V$.  For $e \in \binom{V}{2}$, we say that $f(e)$ is the \emph{weight} of $e$.  A weighted graph is \emph{nonnegative} if every edge has nonnegative weight.  Alternatively, a nonnegative weighted graph on $V$ is a triple $(V,E,f)$, where $G=(V,E)$ is a (simple) graph and $f:E \rightarrow \R_+$ is an assignment of positive reals to the edges; here, it is understood that pairs in $\binom{V}{2}\setminus E$ get weight $0$.  A similar notion may be used in the presence of negative edges.  

Here we assume a finite vertex set, typically $V=[n]:=\{1,2,\dots,n\}$.  The set of weighted graphs forms a vector space of dimension $\binom{n}{2}$ over the reals.  Thus we may identify weighted graphs with vectors in $\R^{\binom{n}{2}}$.  We shall adopt the colexicographic order $\{1,2\}$, $\{1,3\}$, $\{2,3\}$, $\{1,4\}, \dots$ on $\binom{[n]}{2}$, in which edge $e$ precedes edge $f$ if and only if $\max(e \oplus f) \in f$, and index vectors accordingly.  In this way, for $1 \le m \le n$, the prefix of first $\binom{m}{2}$ coordinates in a vector corresponds to the subgraph induced by $\{1,2,\dots,m\}$.

Recall that a \emph{cone} in $\R^d$ is a set $\kappa$ which is closed under both addition and scalar multiplication by nonnegative reals.  The cone \emph{generated by}  $v_1,\dots,v_k$ is $\{\sum_{i=1}^k a_i v_i : a_i \ge 0\}$.  For instance, the set of nonnegative weighted graphs forms a cone corresponding to the nonnegative orthant of $\R^{\binom{n}{2}}$, and hence it is generated by the standard basis.

We are interested here in the cone $\tau_n \subset \R^{\binom{n}{2}}$ of weighted graphs on $n$ vertices generated by triangles.  
A triangle is understood to mean a copy of $K_3$ on $V$, that is, a weighted graph 
assuming the value $1$ on edges in a 3-subset $\{x,y,z\} \subseteq V$ and $0$ otherwise.
The cone $\tau_3$ is simply the ray $(1,1,1)\R$ in $\R^3$.  In the case $n=4$, it is easy to see that a weighted graph is a linear combination of triangles if and only if the sum of weights on any two disjoint edges is a constant.  For $n \ge 5$, the set of all triangles spans $\R^{\binom{n}{2}}$ by linear combinations; see Proposition~\ref{W5} to follow.  It follows that $\tau_n$ has full dimension $\binom{n}{2}$ if $n \ge 5$.

\subsection{Graph decompositions}

An $F$-\emph{decomposition} of a graph $G=(V,E)$  is a partition of its edge set $E$ into (edge sets of) subgraphs, each isomorphic to $F$.  We consider the case $F=K_3$, and also refer to a $K_3$-decomposition as a triangle decomposition.  A triangle decomposition of $K_n$ is equivalent to a Steiner triple system of order $n$, which exists if and only if $n \equiv 1$ or $3 \pmod{6}$.  

A \emph{fractional} $K_3$-\emph{decomposition} of $G$ is an assignment of nonnegative real weights to the triangles in $G$ such that, for every edge $e$ of $G$, the sum of weights assigned to triangles containing $e$ equals $1$.  From our formulation above, $G$ has a fractional $K_3$-decomposition if and only if it belongs to the cone $\tau_n$.  For dense graphs $G$, the recent paper \cite{DP} establishes the current record minimum degree threshold for existence of a fractional triangle decomposition.  Here and in what follows, $\delta(G)$ denotes the minimum degree of $G$.

\begin{thm}[\cite{DP}]
\label{dp}
For sufficiently large $n$, every graph $G$ on $n$ vertices with $\delta(G) > 0.827 n$ belongs to $\tau_n$.
\end{thm}

Barber, K\"uhn, Lo and Osthus showed in \cite{BKLO} showed that a minimum degree threshold sufficient for fractional $K_3$-decomposition is also roughly sufficient for the exact triangle decomposition problem.  This gives considerable motivation to studying degree thresholds for fractional triangle decompositions, and the cone $\tau_n$ in general.  Indeed, reducing $0.827$ to $\frac{3}{4}$ in Theorem~\ref{dp} would nearly establish Conjecture~\ref{nw-conj} except for some cases very close to the boundary.

\subsection{Organization}

We initiate a detailed study of $\tau_n$, especially its facets, and the connection with triangle decompositions of (weighted) graphs.  Section 2 contains some additional background relevant for our problem, including language for polyhedral cones and some reformulations of $\tau_n$. One such related object is the metric polytope, which we briefly discuss in Section~\ref{sec:metric}.
In Section 3, we identify some simple arithmetic constraints on entries of the normal vectors.  Then, in Section 4, we report on a computer-aided classification of facets of $\tau_n$ for $n \le 8$ (this being essentially contained in earlier computations on the metric polytope) and in addition push the computation to a  probably complete classification for $n=9$.  Section 5 contains a `vertex splitting' operation which generates many infinite families of facets. In spite of this partial inductive structure, there is a surprising level of complexity to $\tau_n$.  In Section 6, we examine a class of facets having a hybrid combinatorial-geometric structure.  From these, we produce a new class of graphs having no fractional triangle decomposition but minimum degree approaching $3n/4$ from below.

\section{Background}

\subsection{Cones}

First we review some background on cones.  For our purposes, all cones are assumed to be `polyhedral' (finitely generated).  Unless otherwise specified, cones are `pointed' ($u,-u \in \kappa$ implies $u=0$)  and of full dimension in their vector space.

Let $\kappa$ be a cone.  A {\em face} of $\kappa$ is a cone $\eta \subseteq \kappa$ such that for all 
$u \in \eta$, if $u= u_1 + u_2$ with $u_1,u_2 \in \kappa$, then $u_1,u_2 \in \eta$.
A face of dimension 1 is called an {\em extremal ray} of $\kappa$, while a face of codimension 1 is called a {\em facet} of $\kappa$.  
Two extremal rays are \emph{adjacent} if they span a face of dimension 2.  Likewise, two facets are adjacent if they intersect in a face of codimension 2.  

The discussion from now on focuses on cones in real Euclidean space $\mathbb{R}^m$.  The usual inner product $\langle \cdot , \cdot \rangle$ is used.  
When matrices are involved, we adopt the convention that in $\langle a,b \rangle$, $a$ is a (dual) row vector and $b$ is an (ordinary) column vector.  

A \emph{supporting vector} for a cone $\kappa$ in $\R^m$ is a nonzero vector $y \in \R^m$ such that $\langle y,u \rangle \ge 0$ for all $u \in \kappa$.  From this inequality, $y$ defines a half-space containing $\kappa$.
A result of fundamental importance is that a cone $\kappa$ is the intersection of all half-spaces defined by supporting vectors of $\kappa$. Theorem~\ref{Farkas} below states this in the concrete setting which shall be used herein. 

Given an $m \times n$ matrix $A$, the set cone$(A) = \{ A x : x \in \R^n, x \geq 0 \}$ is a closed and polyhedral cone in $\R^m$.  The dimension of cone$(A)$ is equal to the rank of $A$.  The following well known result provides necessary and sufficient conditions for a point to belong to cone$(A)$.  
\begin{thm}[Farkas Lemma]
\label{Farkas}
Let $A$ be an $m \times n$ matrix over $\R$, and let $b \in \R^m$.
Then $A x = b$ has
a solution $x \ge 0$ if and only if $\langle y, b \rangle \geq 0$ for all $y \in \R^m$ such that $y A \geq 0$.  
\end{thm}

\rks  
One direction of this result is immediate.  Suppose $A x = b$ has a nonnegative solution $x \in \R^n$, and let $y$ be such that $yA \geq 0$. Then $\langle y,b \rangle = \langle y,Ax \rangle = \langle y A, x \rangle \geq 0.$
The converse asserts the existence of a `separating hyperplane' between cone$(A)$ and a point $b \not\in \text{cone}(A)$.

It is enough to check the condition in Theorem~\ref{Farkas} for $y$ corresponding to facets of cone$(A)$. Adapting the simplex algorithm or Fourier-Motzkin elimination gives a procedure to enumerate the facets of cone$(A)$.  Indeed, testing for membership in cone$(A)$ is a linear programming problem whose dual is described by Theorem~\ref{Farkas}.

\subsection{An inclusion matrix}

We return to the setting of (edge-weighted) graphs.
For a simple graph $G$ on vertex set $[n]$, let $\mathds{1}_G$ denote the characteristic vector of $E(G)$ in $\R^{\binom{n}{2}}$, where again coordinates are indexed by $\binom{[n]}{2}$.  That is,
$$\mathds{1}_G(e) = 
\begin{cases}
1 & \text{if $e \in E(G)$}, \\ 0 & \text{otherwise}.
\end{cases}$$
We define $W_n$ as the inclusion matrix of $2$-subsets versus $3$-subsets of $[n]$.  That is, for $e \in \binom{[n]}{2}$ and $f \in \binom{[n]}{3}$,
$$W_n(e,f) = 
\begin{cases}
1 & \text{if } e \subseteq f, \\ 
0 & \text{otherwise}.
\end{cases}$$
Alternatively, $W_n$ is the matrix whose columns are the characteristic vectors of all triangles in $K_n$.
It follows that $\tau_n=\mathrm{cone}(W_n)$.
That is, the existence of a fractional triangle decomposition of $G$ is equivalent to the existence of a nonnegative solution $x$ to $W_n x=\mathds{1}_G$.  

The following simple fact is well-known but is important for our analysis to follow.

\begin{prop}
\label{W5}
The $10 \times 10$ inclusion matrix matrix $W_5$ is invertible with 
\begin{equation}
\label{W5-inverse}
W_5^{-1}(f,e) =
\begin{cases}
1/3 & \text{if } |e \cap f|\in \{0,2\}, \\ 
-1/6 & \text{otherwise}.
\end{cases}
\end{equation}
\end{prop}

\begin{proof}
Let $f,f' \subseteq \{1,\dots,5\}$ with $|f|=|f'|=3$.
The inner product of row $f$ of the claimed inverse \eqref{W5-inverse} and column $f'$ of $W_5$ is computed in cases.  If $f'=f$, the product is $3 \cdot 1/3$.  If $|f' \cap f|=1$ or $2$, exactly one $e \subset f'$ satisfies $|e|=2$ and $|f \cap e| \in \{0,2\}$ (this being $e=f^c = f' \setminus f$ or $e=f \cap f'$, respectively).  In either of these cases, the inner product is $1/3-2 \cdot 1/6 = 0$. 
\end{proof}

Proposition~\ref{W5} is useful for checking whether a set $\mathcal{K}$ of triangles on vertex set $[n]$ spans $\R^{\binom{n}{2}}$ for $n >5$.  If there exists a set $S$ of five points, all of whose triangles can be spanned by $\mathcal{K}$, then $\mathds{1}_e$ can be spanned for each $e \in \binom{S}{2}$.  Applying this with different choices of $S$ can produce `new' triangles in the span.  Note we have identified triangles with their characteristic vectors, a slight abuse of language which is convenient in what follows.

\subsection{Some example facets}
\label{sec:ex-facets}

For results on fractional triangle decompositions, the Farkas lemma motivates a study of the facets of $\tau_n$.  A vector $y \in \mathbb{R}^{\binom{n}{2}}$, or alternatively an edge-weighted graph on $n$ vertices, is normal to a facet of $\tau_n$ if: 
(1) $\langle y,\mathds{1}_K \rangle \ge 0$ for all triangles $K$, and (2) the span of triangles $K$ for which $\langle y,\mathds{1}_K \rangle = 0$ has codimension 1.  We call such vectors `facet normals' in what follows.

\begin{ex}
\label{cut}
Let $n \ge 5$ and suppose $(A,B)$ is a partition of $[n]$ with $|A|,|B| \ge 2$. Consider the vector $y$ defined by
\begin{equation}
\label{cut-defn}
y(e) = 
\begin{cases}
2  & \text{if $e \subseteq \binom{A}{2} \cup \binom{B}{2}$}, \\ 
-1 & \text{otherwise}.
\end{cases}
\end{equation}
It is easy to check that $y$ as defined is nonnegative on triangles.  We show that $y$ is a facet normal of $\tau_n$.  Let $\mathcal{K}_0$ be the set of triangles crossing the partition $(A,B)$.  We have $y$ orthogonal to each triangle in $\mathcal{K}_0$.  Put $\mathcal{K}=\mathcal{K}_0 \cup \{K\}$, where $K$ is a triangle inside (say) $A$.  By Proposition~\ref{W5} and the discussion following it, any $\mathds{1}_e$, $e \in \binom{B}{2}$ is spanned by $\mathcal{K}$; this is seen by considering the five-point-set set $e \cup V(K)$.  It follows, then, that every triangle inside $B$ is spanned by $\mathcal{K}$.  Finally, by choosing three points in $B$ and two in $A$, one has every $\mathds{1}_e$, $e \in \binom{A}{2}$, spanned by $\mathcal{K}$, and this is enough to span the entire space.  It follows that $y$ is a facet normal of $\tau_n$.  

We call $y$ as in \eqref{cut-defn} an $(|A|,|B|)$-\emph{cut}.  Proposition~\ref{vertex-split} provides an alternate verification that such vectors are facet normals of $\tau_n$.  We note that there are exponentially many (in $n$) facets of this type.
\end{ex}

Suppose $4 \mid n$ and recall the graph $G=C_4 \cdot K_{n/4}$ mentioned in Section 1.  There exists an equipartition of the vertices $(A,B)$ so that the number of edges of $G$ within $A$ or $B$ equals $4 \binom{n/4}{2}=n^2/8-n/2$, while the number of edges crossing the partition equals $n^2/4$.  Taking $y$ to be the $(n/2,n/2)$-cut facet defined by $(A,B)$, we see that $\langle y, \mathds{1}_G \rangle < 0$.  In fact, this same cut witnesses many other graphs with minimum degree near $3n/4$ also failing to have a (fractional) triangle decomposition.  

\begin{ex}
\label{trivial}
Let $n \ge 6$.  For any $e \in \binom{n}{2}$, the vector $y = \mathds{1}_e$ is a facet normal of $\tau_n$, since $\dim y^\perp = \binom{n}{2}-1$.  We call (positive multiples of) such $y$ and their corresponding facets \emph{trivial}.
\end{ex}

In practice, it is simple to check whether a vector $y$ supports our cone.  For it to be a facet normal, it must be `indecomposable': if $y=y_1+y_2$ with $y_1$ and $y_2$ both supporting vectors, then $y_1=cy$ and $y_2=(1-c)y$ for some $c \in [0,1]$.  This, together with Example~\ref{trivial} leads to an easy observation.

\begin{lemma}
\label{zero-sum}
Let $y$ be a nontrivial facet normal of $\tau_n$.  Then every edge in $\binom{[n]}{2}$ is contained in a triangle $K$ such that $\langle y,\mathds{1}_K \rangle=0$.
\end{lemma}

\begin{proof}
This can be shown by direct verification for $n = 5$ using that the only facets of $\tau_5$ are $(2,3)$-cuts; see Proposition~\ref{W5}.  Suppose $n \ge 6$ and $y$ is a facet normal of $\tau_n$.  If $e$ is an edge in no such triangle, then it is possible to decrease the weight of $e$ in $y$ such that the resulting vector still supports $\tau_n$.  Therefore, $y$ is a positive multiple of $\mathds{1}_e$.
\end{proof}

\begin{ex}
\label{star}
Let $n \ge 6$.  The edge-weighted graph $y$ on vertex set $[n]$ with 
$$y(e) = 
\begin{cases}
-1  & \text{if $e=\{1,n\}$}, \\ 
1  & \text{if $e=\{i,n\}$ for $i \in \{2,\dots,n-1\}$}, \\ 
0 & \text{otherwise}
\end{cases}$$
is a facet normal of $\tau_n$.  This is easy to verify directly.  Any triangle avoiding vertex $n$ is orthogonal to $y$, as is any triangle containing edge $\{1,n\}$.
If we include with these triangles one of positive weight, say on $\{2,3,n\}$, it is possible to span the whole space by a similar argument as in Example~\ref{cut}.  For more details, including why $n \ge 6$ is needed, we refer the reader to \cite{Kseniya}.
\end{ex}

It is natural to call the facet $y$ described by Example~\ref{star} a \emph{star} facet of $\tau_n$.  It has a combinatorial interpretation for decomposition into triangles, albeit of very mild significance: the inequality $\langle y, \mathds{1}_K \rangle \ge 0$ is asserting that, should $G$ have a fractional triangle decomposition, there cannot exist any vertex of degree 1 in $G$.   The star $y$ centered at such a vertex and aligned so that the negative edge is the pendant edge would have $\langle y, \mathds{1}_G \rangle = -1$.

\subsection{Symmetry}

Consider the natural action of the symmetric group $\mathcal{S}_n$ on $\R^{\binom{[n]}{2}}$.  For an edge-weighted graph $y$ and permutation $\alpha \in \mathcal{S}_n$, we have $y^\alpha (e) = y(\alpha^{-1} e)$; that is, the action is induced on edges (edge weights) by permutations of the vertices.

Let us define the \emph{stabilizer} of $y$ to be stab$(y)=\{\alpha \in \mathcal{S}_n: y^\alpha = y \}$.  By the orbit-stabilizer theorem, the number of distinct edge-weighted graphs isomorphic to $y$ on $n$ vertices is $n!/|\text{stab}(y)|$.

Suppose $G$ has automorphism group $\Gamma$.  Testing whether $G$ is in the cone $\tau_n$ amounts to checking $\langle \overline{y}, \mathds{1}_G \rangle \ge 0$ on all $\overline{y}$ of the form $$\overline{y}=\frac{1}{|\Gamma|} \sum_{\alpha \in \Gamma} y^\alpha$$
for some facet normal $y$ of $\tau_n$.   The set of weighted graphs invariant under $\Gamma$ is a subspace of $\R^{\binom{[n]}{2}}$.  Therefore, its intersection with $\tau_n$ is a sub-cone. For example, if  $\Gamma=\mathcal{S}_a \times \mathcal{S}_b$ for $n=a+b$, then the $(a,b)$-cuts (and the nonnegative orthant) give a complete description; see \cite{DW}.   It would be interesting to study invariant sub-cones for other specific groups $\Gamma \le \mathcal{S}_n$.  

\subsection{A matrix formulation}

As an alternative to placing edge-weighted graphs in correspondence with $\mathbb{R}^{\binom{n}{2}}$, we can use the $n \times n$ symmetric matrices with zero diagonal.  Under this slight change in notation, $\tau_n$ is equivalent to the cone generated by the $\binom{n}{3}$ matrices 
\begin{equation}
\label{matrix-cone}
P^\top  \begin{bmatrix}
0 & 1 & 1 \\
1 & 0 & 1 \\
1 & 1 & 0 \\
\end{bmatrix} P \in \R^{n \times n},
\end{equation}
where $P$ is comprised of three rows of an $n \times n$ permutation matrix.
This alternate presentation has some advantages. For example, the characteristic polynomial $\chi_{y}(t)$ of the matrix corresponding to an edge-weighted graph $y$ is preserved under vertex permutation, making it a useful invariant.  Moreover, the factorization of $\chi_{y}(t)$ in $\Q[t]$ carries algebraic information, such as the stabilizer of $y$.

In this context, a graph $G$ has a fractional triangle decomposition if and only if its adjacency matrix $A_G$ is a conical combination of matrices of the form \eqref{matrix-cone}.
A referee suggested a nice variant on this: $G$ admits a fractional triangle-decomposition if and only if its Laplacian matrix $L_G:=D_G-A_G$, where $D_G$ is the diagonal matrix of degrees in $G$, is a conical combination of matrices of the form
\begin{equation}
\label{psd-cone}
P^\top  \left[ \begin{array}{rrr}
2 & -1 & -1 \\
-1 & 2 & -1 \\
-1 & -1 & 2 \\
\end{array} \right] P \in \R^{n \times n},
\end{equation}
where $P$ is as before.  A feature of this reformulation is that both $L_G$ and the matrices in \eqref{psd-cone} are positive semidefinite.

\subsection{The metric polytope}
\label{sec:metric}

The cone $\tau_n$ appears `locally' inside a well-studied polytope.  Recall that a \emph{metric} $d$ on a set $X$ is a function $d: X \times X \rightarrow \R$ such that, for all $x,y,z\in X$,
\vspace{-11pt}
\begin{enumerate}
\item $d(x, y) \geq 0$,
\item[(a$'$)] $d(x, y) = 0$   if and only if  $x = y$,
\item $d(x, y) = d(y, x)$, 
\item $d(x, z) \leq d(x, y) + d(y, z)$. 
\end{enumerate}
\vspace{-11pt}
A \emph{semi-metric} is a function that satisfies only conditions (a), (b) and (c). 
The \emph{metric cone} $\text{Met}_n$ consists of all semi-metrics on an $n$-set.  If we bound $\text{Met}_n$ by considering only those semi-metrics which satisfy the `perimeter inequalities'
\begin{equation}
\label{perim}
d(x, y) + d(x, z) + d(y,z) \le 2
\end{equation}
for all $3$-subsets $\{x,y,z\} \subseteq [n]$, then we obtain the \emph{metric polytope} $\text{met}_n$. Each of the triangle inequalities and perimeter inequalities defines a half-space bounding $\text{met}_n$.  Taking $X=[n]$ and writing $d_{ij}$ for $d(i,j)$, we can consider the metric $d$ as a vector $(d_{12},d_{13},d_{23},\dots,d_{n-1,n})$ in $\R^{\binom{n}{2}}$, and thus embed $\text{Met}_n$ and $\text{met}_n$ in $\R^{\binom{n}{2}}$.  

The point $(2/3,\dots,2/3)$ is a vertex of $\text{met}_n$ for $n>3$ since the vector obtained by incrementing any coordinate by $\epsilon>0$ and simultaneously decrementing a vertex-disjoint coordinate by $\epsilon$ fails to satisfy (\ref{perim}).  Near this vertex (within $2/9$ in the box norm) the triangle inequalities are automatically satisfied.  It follows that
$(2/3,\dots,2/3)-\text{met}_n$, near the origin, is described by the inequalities $a_{ij}+a_{ik}+a_{jk} \ge 0$ for $3$-subsets $\{i,j,k\}$.
The following result is an immediate consequence.

\begin{prop}
\label{met-n}
The halfspaces bounding $\tau_n$ admit a natural bijection with the edges of $\text{met}_n$ incident with $(2/3,\dots,2/3)$.
\end{prop}

\section{Arithmetic and combinatorial structure}

In this section, we sample some structure forced on facets of $\tau_n$, specifically on the entries of their normal vectors.
We begin with an easy observation on the relative sizes of extreme entries in facet normals.

\begin{prop}
\label{signs-mags}
Let $y$ be a nontrivial facet normal of $\tau_n$.  Let $a$ and $b$ denote, respectively, the maximum and minimum entry in $y$.  Then we have $-b/2 \le a \le -2b$.
\end{prop}

\begin{proof}
First, suppose $a \ge -b$.  Using Lemma~\ref{zero-sum}, choose a triangle $K$ in $y^\perp$ containing an edge of weight $a$.  The other two edges of $K$ must have negative weight, since $a$ is largest among the entries in magnitude.  It follows that one of these other edges in $K$ has weight at most $-a/2$.  This establishes $a \le -2b$.  The case $a < -b$ is similar.
\end{proof}

\rks
It is easy to see that $a=1$, $b=-2$ is not possible for facets (although it is possible for supporting vectors).  Indeed, it may be the case that the first inequality can be strengthened to $a>-b/2$ or even $a \ge -b$.

Presumably, Proposition~\ref{signs-mags} only scratches the surface of constraints on the signs and relative magnitudes in facet normals.  We do not explore this further here, although it seems reasonable to guess that entries have some central tendency and are roughly symmetric about their mean.

\begin{prop}
\label{bipG}
Let $y$ be a facet normal of $\tau_n$ and let $G$ be the simple graph carrying the nonpositive entries of $y$.  If $G$ is bipartite, then it is complete bipartite.
\end{prop}

\begin{proof}
Suppose $(A,B)$ is a bipartition of the vertices of $G$ such that all edges in $\binom{A}{2}$ and $\binom{B}{2}$ are positive. Let $z$ be the cut facet corresponding to $(A,B)$; see Example~\ref{cut}.  Since all entries of $y$ on edges within $A$ and $B$ are positive, it follows that for some $\epsilon>0$, $y-\epsilon z$ also supports $\tau_n$.  This is a contradiction to the indecomposability of $y$.
\end{proof}

Next, we offer a purely arithmetic constraint.  Let us say that a vector is in \emph{standard form} if its entries are integers with greatest common divisor equal to 1.

\begin{prop}
\label{mod3}
Let $n \ge 5$ and suppose $y$ is a facet normal of $\tau_n$ in standard form.  The following are equivalent:
\begin{enumerate}
\item
$\langle y,\mathds{1}_K \rangle \equiv 0 \pmod{3}$ for all triangles $K$;
\item
no entry of $y$ is $0 \pmod{3}$; and
\item
all entries of $y$ are $1 \pmod{3}$ or all entries of $y$ are $2 \pmod{3}$.
\end{enumerate}
\end{prop}

\begin{proof}
The implication (c) $\Rightarrow$ (a) is obvious, so it is enough to prove (a) $\Rightarrow$ (b) and (b) $\Rightarrow$ (c). 

Suppose $y$ satisfies (a) but that $y_e = 0$ for some edge $e=\{u,v\}$.  For every triangle involving $e$, the entries of $y$ on the other two edges must be different and nonzero (mod 3), using (a).  Since $n \ge 5$, there are at least three such triangles.  It follows that there is a $4$-cycle, say $w_1,u,w_2,v$ in which the entries of $y$ on its edges are $1,1,2,2 \pmod{3}$, in this order.  There is now no possible value for $y$ on $w_1w_2$: it must be $1 \pmod{3}$ from the triangle with $u$ and $2 \pmod{3}$ from the triangle with $v$.  This contradiction shows (a) $\Rightarrow$ (b).

For the next implication, suppose $y$ satisfies (b).  For $i=1,2$, let $E_i$ be the set of edges $e$ such that $y_e \equiv i \pmod{3}$.  
Let $\mathcal{K}$ be the set of all triangles $K$ such that $\langle y,\one_K \rangle = 0$.  Any such triangle must have all three edges in $E_1$ or all three edges in $E_2$, by (b).  Let $\mathcal{K}_i$ be the set of triangles $K$ such that $\langle y,\one_K \rangle = 0$ and all three edges belong to $E_i$.    
Now, take any edge $e$, and assume without loss of generality that $e \in E_1$.  Since $y$ is a facet-normal and $y_e \neq 0$, we have that $\{\one_K:K \in \mathcal{K}\} \cup \{\one_e\}$ spans $\R^{E_1 \cup E_2}$.  Examining the coordinates indexed by $E_2$, it follows that $\{\one_K:K \in \mathcal{K}_2\}$ spans $\{0\} \oplus \R^{E_2}$.  Letting $y_2$ denote the projection of $y$ onto $\R^{E_2}$, we have $\dim \{y_2\}^\perp = |E_2|$.  It follows that $y_2 = 0$, and this implies $E_2=\emptyset$.  This proves $y$ is constant (mod 3) and we have shown (c).
\end{proof}

As a result of Proposition~\ref{mod3}, we can partition all facets of $\tau_n$ into three categories according to their normal vectors in standard form.  It is natural to label these $0,1,2$ according to whether some zero is present, all entries are $1 \pmod{3}$, or all entries are $2 \pmod{3}$, respectively.
The trivial and star facets belong to category $0$.  The cut facets belong to category $2$.  Example~\ref{1mod3} to follow gives the first example of a facet in category $1$, which occurs for $n=7$.

The following example shows that Proposition~\ref{mod3} is somewhat tight.

\begin{ex}
The vector
$$\begin{array}{rrrrrrrrr}
(10,\\ 
8, & 6, \\
2, & 4, & 2, \\
2, & 4, & 2, & -4, \\
2, & 2, & -4, & 2, & 2, \\
-1, & -1, & 5, & 5, & -1, & -1, \\
-1, & -1, & 5, & 5, & -1, & -1, & 2, \\
-4, & -2, & -4, & 2, & 2, & 8, & 5, & 5, \\
-4, & -6, & 0, & 2, & 2, & 4, & 7, & 7, & 8) \\
\end{array}
$$  
\normalsize
is in standard form and normal to a facet of $\tau_{10}$.  There are only three entries congruent to $0 \pmod{3}$, namely $6,0,-6$; these are supported on edges joined in a path.
\end{ex}


\section{Classification for small $n$}

The case $n=5$ is the first in which $\tau_n$ has full dimension.  As a direct consequence of Proposition~\ref{W5}, the only facets of $\tau_5$ are the 
$(2,3)$-cuts, giving a total of 10 facets.  Moreover, their structure is `simplicial' (any two are adjacent) since the ten facet normals are linearly independent in $\R^{10}$.  In this section we enumerate facets (and the adjacent facets for each) for the next few values of $n$.
Small cases can be easily replicated using, e.g., the built-in function {\tt Cone()}  in Sage \cite{Sage}.  We remark that much of the enumerative work to follow has been done in the context of the metric cone, \cite{Deza}.


A classification of facets of $\tau_6$ is shown in Table~\ref{f6}.  We note that the examples in Section~\ref{sec:ex-facets}, namely trivial, star, $(3,3)$-cut, and $(2,4)$-cut, give a complete description of $\tau_6$.  In the table, each row gives a vector, with coordinates corresponding to the colex order on $\binom{[6]}{2}$ for an isomorphism class. The column with heading $\#$ gives the number of distinct copies induced under the action of $\mathcal{S}_6$.  The rightmost column gives the degree, or number of facets to which the corresponding facet is adjacent.
The three nontrivial isomorphism types are displayed as weighted graphs in Figure~\ref{f6fig}.  A $\{$blue,red$\}$-edge-coloring illustrates the positive and negative weights, with magnitudes as labeled.

\begin{table}[htpb]
\begin{tabular}{rlrr}
& representative & $\#$ & deg \\
\hline
1. & (1, 0, 0, 0, 0, 0, 0, 0, 0, 0, 0, 0, 0, 0, 0) & 15 & 32\\
2. & (1, 1, 0, 1, 0, 0, 1, 0, 0, 0, -1, 0, 0, 0, 0) & 30 & 14\\
3. & (2, 2, 2, -1, -1, -1, -1, -1, -1, 2, -1, -1, -1, 2, 2) & 10 & 57 \\
4. & (2, 2, 2, 2, 2, 2, -1, -1, -1, -1, -1, -1, -1, -1, 2) & 15 & 32\\
\hline
& total & 70
\end{tabular}
\caption{Isomorphism classes of facet normals of $\tau_6$, in standard form}
\label{f6}
\end{table}

\begin{figure}[htpb]
\begin{minipage}{.32\linewidth}
\begin{center}
\begin{tikzpicture}
\node at ( 1.732 , 1.0 ){$ 1 $};
\draw[blue! 100.0 !red] ( 2.0 , 0.0 )--( 1.0 , 1.732 );
\draw[blue! 100.0 !red] ( 2.0 , 0.0 )--( -1.0 , 1.732 );
\draw[blue! 100.0 !red] ( 2.0 , 0.0 )--( -2.0 , 0.0 );
\draw[blue! 100.0 !red] ( 2.0 , 0.0 )--( -1.0 , -1.732 );
\node at ( 1.732 , -1.0 ){$ -1 $};
\draw[red! 100.0 !blue,dashed] ( 2.0 , 0.0 )--( 1.0 , -1.732 );
\filldraw ( 2.0 , 0.0 ) circle [radius=.1];
\filldraw ( 1.0 , 1.732 ) circle [radius=.1];
\filldraw ( -1.0 , 1.732 ) circle [radius=.1];
\filldraw ( -2.0 , 0.0 ) circle [radius=.1];
\filldraw ( -1.0 , -1.732 ) circle [radius=.1];
\filldraw ( 1.0 , -1.732 ) circle [radius=.1];
\end{tikzpicture}
\end{center}
\end{minipage}
\begin{minipage}{.32\linewidth}
\begin{center}
\begin{tikzpicture}
\node at ( 1.732 , 1.0 ){$ 2 $};
\draw[blue! 100.0 !red] ( 2.0 , 0.0 )--( 1.0 , 1.732 );
\draw[blue! 100.0 !red] ( 2.0 , 0.0 )--( -1.0 , 1.732 );
\draw[blue! 100.0 !red] ( 1.0 , 1.732 )--( -1.0 , 1.732 );
\draw[red! 100.0 !blue,dashed] ( 2.0 , 0.0 )--( -2.0 , 0.0 );
\draw[red! 100.0 !blue,dashed] ( 1.0 , 1.732 )--( -2.0 , 0.0 );
\node at ( -1.732 , 1.0 ){$ -1 $};
\draw[red! 100.0 !blue,dashed] ( -1.0 , 1.732 )--( -2.0 , 0.0 );
\draw[red! 100.0 !blue,dashed] ( 2.0 , 0.0 )--( -1.0 , -1.732 );
\draw[red! 100.0 !blue,dashed] ( 1.0 , 1.732 )--( -1.0 , -1.732 );
\draw[red! 100.0 !blue,dashed] ( -1.0 , 1.732 )--( -1.0 , -1.732 );
\draw[blue! 100.0 !red] ( -2.0 , 0.0 )--( -1.0 , -1.732 );
\draw[red! 100.0 !blue,dashed] ( 2.0 , 0.0 )--( 1.0 , -1.732 );
\draw[red! 100.0 !blue,dashed] ( 1.0 , 1.732 )--( 1.0 , -1.732 );
\draw[red! 100.0 !blue,dashed] ( -1.0 , 1.732 )--( 1.0 , -1.732 );
\draw[blue! 100.0 !red] ( -2.0 , 0.0 )--( 1.0 , -1.732 );
\draw[blue! 100.0 !red] ( -1.0 , -1.732 )--( 1.0 , -1.732 );
\filldraw ( 2.0 , 0.0 ) circle [radius=.1];
\filldraw ( 1.0 , 1.732 ) circle [radius=.1];
\filldraw ( -1.0 , 1.732 ) circle [radius=.1];
\filldraw ( -2.0 , 0.0 ) circle [radius=.1];
\filldraw ( -1.0 , -1.732 ) circle [radius=.1];
\filldraw ( 1.0 , -1.732 ) circle [radius=.1];
\end{tikzpicture}
\end{center}
\end{minipage}
\begin{minipage}{.32\linewidth}
\begin{center}
\begin{tikzpicture}
\node at ( 1.732 , 1.0 ){$ 2 $};
\draw[blue! 100.0 !red] ( 2.0 , 0.0 )--( 1.0 , 1.732 );
\draw[blue! 100.0 !red] ( 2.0 , 0.0 )--( -1.0 , 1.732 );
\draw[blue! 100.0 !red] ( 1.0 , 1.732 )--( -1.0 , 1.732 );
\draw[blue! 100.0 !red] ( 2.0 , 0.0 )--( -2.0 , 0.0 );
\draw[blue! 100.0 !red] ( 1.0 , 1.732 )--( -2.0 , 0.0 );
\draw[blue! 100.0 !red] ( -1.0 , 1.732 )--( -2.0 , 0.0 );
\draw[red! 100.0 !blue,dashed] ( 2.0 , 0.0 )--( -1.0 , -1.732 );
\draw[red! 100.0 !blue,dashed] ( 1.0 , 1.732 )--( -1.0 , -1.732 );
\draw[red! 100.0 !blue,dashed] ( -1.0 , 1.732 )--( -1.0 , -1.732 );
\node at ( -1.732 , -1.0 ){$ -1 $};
\draw[red! 100.0 !blue,dashed] ( -2.0 , 0.0 )--( -1.0 , -1.732 );
\draw[red! 100.0 !blue,dashed] ( 2.0 , 0.0 )--( 1.0 , -1.732 );
\draw[red! 100.0 !blue,dashed] ( 1.0 , 1.732 )--( 1.0 , -1.732 );
\draw[red! 100.0 !blue,dashed] ( -1.0 , 1.732 )--( 1.0 , -1.732 );
\draw[red! 100.0 !blue,dashed] ( -2.0 , 0.0 )--( 1.0 , -1.732 );
\draw[blue! 100.0 !red] ( -1.0 , -1.732 )--( 1.0 , -1.732 );
\filldraw ( 2.0 , 0.0 ) circle [radius=.1];
\filldraw ( 1.0 , 1.732 ) circle [radius=.1];
\filldraw ( -1.0 , 1.732 ) circle [radius=.1];
\filldraw ( -2.0 , 0.0 ) circle [radius=.1];
\filldraw ( -1.0 , -1.732 ) circle [radius=.1];
\filldraw ( 1.0 , -1.732 ) circle [radius=.1];
\end{tikzpicture}
\end{center}
\end{minipage}
\caption{Weighted graphs for the nontrivial facet normals of $\tau_6$}
\label{f6fig}
\end{figure}

The following result is a straightforward consequence of the classification.

\begin{prop}
A multigraph on six vertices admits a fractional triangle decomposition if and only if it satisfies all inequalities corresponding to cuts and its underlying simple graph has no pendant vertices.
\end{prop}

%




For $n=7$, the results appear in Table~\ref{f7} with headings as before and weighted graphs in Figure~\ref{f7fig}.

\begin{table}[htpb]
\begin{tabular}{rlrr}
& representative & $\#$ & deg \\
\hline
1. & (1, 0, 0, 0, 0, 0, 0, 0, 0, 0, 0, 0, 0, 0, 0, 0, 0, 0, 0, 0, 0) & 21 & 340 \\
2. & (1, 1, 0, 1, 0, 0, 1, 0, 0, 0, 0, 1, 1, 1, -1, -1, 0, 0, 0, 0, 1) & 420 & 20\\
3. & (1, 1, 0, 1, 0, 0, 1, 0, 0, 0, 1, 0, 0, 0, 0, -1, 0, 0, 0, 0, 0) & 42 & 75 \\
4. & (2, 0, 0, 0, 0, 0, 0, 0, 0, 0, 0, 0, 0, 0, 0, -1, -1, 1, 1, 1, 1) & 105 & 75 \\
5. & (2, 2, 2, 2, 2, 2, -1, -1, -1, -1, -1, -1, -1, -1, 2, -1, -1, -1, -1, 2, 2) & 35 & 340\\
6. & (2, 2, 2, 2, 2, 2, 2, 2, 2, 2, -1, -1, -1, -1, -1, -1, -1, -1, -1, -1, 2) & 21 & 75\\
7. & (4, 4, -2, 1, 1, 1, 1, 1, 1, -2, -2, 4, -2, 1, 1, -2, -2, 4, 1, 1, 4) & 252 & 20\\
\hline
& total & 896
\end{tabular}
\caption{Isomorphism classes of facet normals of $\tau_7$, in standard form}
\label{f7}
\end{table}

\begin{figure}[htbp]
\begin{minipage}{.32\linewidth}
\begin{center}
\begin{tikzpicture}[scale=.75]
\node at ( 1.802 , 0.868 ){$ 1 $};
\draw[blue! 100.0 !red] ( 2.0 , 0.0 )--( 1.247 , 1.564 );
\draw[blue! 100.0 !red] ( 2.0 , 0.0 )--( -0.445 , 1.95 );
\draw[blue! 100.0 !red] ( 2.0 , 0.0 )--( -1.802 , 0.868 );
\draw[blue! 100.0 !red] ( 2.0 , 0.0 )--( -1.802 , -0.868 );
\draw[blue! 100.0 !red] ( 1.247 , 1.564 )--( -0.445 , -1.95 );
\draw[blue! 100.0 !red] ( -0.445 , 1.95 )--( -0.445 , -1.95 );
\draw[blue! 100.0 !red] ( -1.802 , 0.868 )--( -0.445 , -1.95 );
\node at ( -1.35 , -1.65 ){$ -1 $};
\draw[red! 100.0 !blue,dashed] ( -1.802 , -0.868 )--( -0.445 , -1.95 );
\draw[red! 100.0 !blue,dashed] ( 2.0 , 0.0 )--( 1.247 , -1.564 );
\draw[blue! 100.0 !red] ( -0.445 , -1.95 )--( 1.247 , -1.564 );
\filldraw ( 2.0 , 0.0 ) circle [radius=.1];
\filldraw ( 1.247 , 1.564 ) circle [radius=.1];
\filldraw ( -0.445 , 1.95 ) circle [radius=.1];
\filldraw ( -1.802 , 0.868 ) circle [radius=.1];
\filldraw ( -1.802 , -0.868 ) circle [radius=.1];
\filldraw ( -0.445 , -1.95 ) circle [radius=.1];
\filldraw ( 1.247 , -1.564 ) circle [radius=.1];
\end{tikzpicture}
\end{center}
\end{minipage}
\begin{minipage}{.32\linewidth}
\begin{center}
\begin{tikzpicture}[scale=.75]
\node at ( 1.802 , 0.868 ){$ 1 $};
\draw[blue! 100.0 !red] ( 2.0 , 0.0 )--( 1.247 , 1.564 );
\draw[blue! 100.0 !red] ( 2.0 , 0.0 )--( -0.445 , 1.95 );
\draw[blue! 100.0 !red] ( 2.0 , 0.0 )--( -1.802 , 0.868 );
\draw[blue! 100.0 !red] ( 2.0 , 0.0 )--( -1.802 , -0.868 );
\draw[blue! 100.0 !red] ( 2.0 , 0.0 )--( -0.445 , -1.95 );
\node at ( 1.85 , -0.9 ){$ -1 $};
\draw[red! 100.0 !blue,dashed] ( 2.0 , 0.0 )--( 1.247 , -1.564 );
\filldraw ( 2.0 , 0.0 ) circle [radius=.1];
\filldraw ( 1.247 , 1.564 ) circle [radius=.1];
\filldraw ( -0.445 , 1.95 ) circle [radius=.1];
\filldraw ( -1.802 , 0.868 ) circle [radius=.1];
\filldraw ( -1.802 , -0.868 ) circle [radius=.1];
\filldraw ( -0.445 , -1.95 ) circle [radius=.1];
\filldraw ( 1.247 , -1.564 ) circle [radius=.1];
\end{tikzpicture}
\end{center}
\end{minipage}
\begin{minipage}{.32\linewidth}
\begin{center}
\begin{tikzpicture}[scale=.75]
\node at ( 1.802 , 0.868 ){$ 2 $};
\draw[blue! 100.0 !red, line width=0.32mm] ( 2.0 , 0.0 )--( 1.247 , 1.564 );
\node at ( 1.85 , -0.9 ){$ -1 $};
\draw[red! 100.0 !blue,dashed] ( 2.0 , 0.0 )--( 1.247 , -1.564 );
\draw[red! 100.0 !blue,dashed] ( 1.247 , 1.564 )--( 1.247 , -1.564 );
\draw[blue! 100 !red] ( -0.445 , 1.95 )--( 1.247 , -1.564 );
\draw[blue! 100 !red] ( -1.802 , 0.868 )--( 1.247 , -1.564 );
\draw[blue! 100 !red] ( -1.802 , -0.868 )--( 1.247 , -1.564 );
\node at ( 0.445 , -1.95 ){$ 1 $};
\draw[blue! 100 !red] ( -0.445 , -1.95 )--( 1.247 , -1.564 );
\filldraw ( 2.0 , 0.0 ) circle [radius=.1];
\filldraw ( 1.247 , 1.564 ) circle [radius=.1];
\filldraw ( -0.445 , 1.95 ) circle [radius=.1];
\filldraw ( -1.802 , 0.868 ) circle [radius=.1];
\filldraw ( -1.802 , -0.868 ) circle [radius=.1];
\filldraw ( -0.445 , -1.95 ) circle [radius=.1];
\filldraw ( 1.247 , -1.564 ) circle [radius=.1];
\end{tikzpicture}
\end{center}
\end{minipage}

\bigskip

\begin{minipage}{.32\linewidth}
\begin{center}
\begin{tikzpicture}[scale=.75]
\node at ( 1.802 , 0.868 ){$ 2 $};
\draw[blue! 100.0 !red] ( 2.0 , 0.0 )--( 1.247 , 1.564 );
\draw[blue! 100.0 !red] ( 2.0 , 0.0 )--( -0.445 , 1.95 );
\draw[blue! 100.0 !red] ( 1.247 , 1.564 )--( -0.445 , 1.95 );
\draw[blue! 100.0 !red] ( 2.0 , 0.0 )--( -1.802 , 0.868 );
\draw[blue! 100.0 !red] ( 1.247 , 1.564 )--( -1.802 , 0.868 );
\draw[blue! 100.0 !red] ( -0.445 , 1.95 )--( -1.802 , 0.868 );
\draw[red! 100.0 !blue,dashed] ( 2.0 , 0.0 )--( -1.802 , -0.868 );
\draw[red! 100.0 !blue,dashed] ( 1.247 , 1.564 )--( -1.802 , -0.868 );
\draw[red! 100.0 !blue,dashed] ( -0.445 , 1.95 )--( -1.802 , -0.868 );
\node at ( -2.1 , 0.0 ){$ -1 $};
\draw[red! 100.0 !blue,dashed] ( -1.802 , 0.868 )--( -1.802 , -0.868 );
\draw[red! 100.0 !blue,dashed] ( 2.0 , 0.0 )--( -0.445 , -1.95 );
\draw[red! 100.0 !blue,dashed] ( 1.247 , 1.564 )--( -0.445 , -1.95 );
\draw[red! 100.0 !blue,dashed] ( -0.445 , 1.95 )--( -0.445 , -1.95 );
\draw[red! 100.0 !blue,dashed] ( -1.802 , 0.868 )--( -0.445 , -1.95 );
\draw[blue! 100.0 !red] ( -1.802 , -0.868 )--( -0.445 , -1.95 );
\draw[red! 100.0 !blue,dashed] ( 2.0 , 0.0 )--( 1.247 , -1.564 );
\draw[red! 100.0 !blue,dashed] ( 1.247 , 1.564 )--( 1.247 , -1.564 );
\draw[red! 100.0 !blue,dashed] ( -0.445 , 1.95 )--( 1.247 , -1.564 );
\draw[red! 100.0 !blue,dashed] ( -1.802 , 0.868 )--( 1.247 , -1.564 );
\draw[blue! 100.0 !red] ( -1.802 , -0.868 )--( 1.247 , -1.564 );
\draw[blue! 100.0 !red] ( -0.445 , -1.95 )--( 1.247 , -1.564 );
\filldraw ( 2.0 , 0.0 ) circle [radius=.1];
\filldraw ( 1.247 , 1.564 ) circle [radius=.1];
\filldraw ( -0.445 , 1.95 ) circle [radius=.1];
\filldraw ( -1.802 , 0.868 ) circle [radius=.1];
\filldraw ( -1.802 , -0.868 ) circle [radius=.1];
\filldraw ( -0.445 , -1.95 ) circle [radius=.1];
\filldraw ( 1.247 , -1.564 ) circle [radius=.1];
\end{tikzpicture}
\end{center}
\end{minipage}
\begin{minipage}{.32\linewidth}
\begin{center}
\begin{tikzpicture}[scale=.75]
\node at ( 1.802 , 0.868 ){$ 2 $};
\draw[blue! 100.0 !red] ( 2.0 , 0.0 )--( 1.247 , 1.564 );
\draw[blue! 100.0 !red] ( 2.0 , 0.0 )--( -0.445 , 1.95 );
\draw[blue! 100.0 !red] ( 1.247 , 1.564 )--( -0.445 , 1.95 );
\draw[blue! 100.0 !red] ( 2.0 , 0.0 )--( -1.802 , 0.868 );
\draw[blue! 100.0 !red] ( 1.247 , 1.564 )--( -1.802 , 0.868 );
\draw[blue! 100.0 !red] ( -0.445 , 1.95 )--( -1.802 , 0.868 );
\draw[blue! 100.0 !red] ( 2.0 , 0.0 )--( -1.802 , -0.868 );
\draw[blue! 100.0 !red] ( 1.247 , 1.564 )--( -1.802 , -0.868 );
\draw[blue! 100.0 !red] ( -0.445 , 1.95 )--( -1.802 , -0.868 );
\draw[blue! 100.0 !red] ( -1.802 , 0.868 )--( -1.802 , -0.868 );
\draw[red! 100.0 !blue,dashed] ( 2.0 , 0.0 )--( -0.445 , -1.95 );
\draw[red! 100.0 !blue,dashed] ( 1.247 , 1.564 )--( -0.445 , -1.95 );
\draw[red! 100.0 !blue,dashed] ( -0.445 , 1.95 )--( -0.445 , -1.95 );
\draw[red! 100.0 !blue,dashed] ( -1.802 , 0.868 )--( -0.445 , -1.95 );
\node at ( -1.35 , -1.65 ){$ -1 $};
\draw[red! 100.0 !blue,dashed] ( -1.802 , -0.868 )--( -0.445 , -1.95 );
\draw[red! 100.0 !blue,dashed] ( 2.0 , 0.0 )--( 1.247 , -1.564 );
\draw[red! 100.0 !blue,dashed] ( 1.247 , 1.564 )--( 1.247 , -1.564 );
\draw[red! 100.0 !blue,dashed] ( -0.445 , 1.95 )--( 1.247 , -1.564 );
\draw[red! 100.0 !blue,dashed] ( -1.802 , 0.868 )--( 1.247 , -1.564 );
\draw[red! 100.0 !blue,dashed] ( -1.802 , -0.868 )--( 1.247 , -1.564 );
\draw[blue! 100.0 !red] ( -0.445 , -1.95 )--( 1.247 , -1.564 );
\filldraw ( 2.0 , 0.0 ) circle [radius=.1];
\filldraw ( 1.247 , 1.564 ) circle [radius=.1];
\filldraw ( -0.445 , 1.95 ) circle [radius=.1];
\filldraw ( -1.802 , 0.868 ) circle [radius=.1];
\filldraw ( -1.802 , -0.868 ) circle [radius=.1];
\filldraw ( -0.445 , -1.95 ) circle [radius=.1];
\filldraw ( 1.247 , -1.564 ) circle [radius=.1];
\end{tikzpicture}
\end{center}
\end{minipage}
\begin{minipage}{.32\linewidth}
\begin{center}
\begin{tikzpicture}[scale=.75]
\node at ( 1.802 , 0.868 ){$ 4 $};
\draw[blue! 100.0 !red,line width=0.45mm] ( 2.0 , 0.0 )--( 1.247 , 1.564 );
\draw[red! 100.0 !blue,dashed] ( 2.0 , 0.0 )--( -0.445 , 1.95 );
\node at (-2.3 , 0 ){$ -2 $};
\draw[blue! 100.0 !red,line width=0.45mm] ( 1.247 , 1.564 )--( -0.445 , 1.95 );
\draw[blue! 100 !red] ( 2.0 , 0.0 )--( -1.802 , 0.868 );
\draw[blue! 100 !red] ( 1.247 , 1.564 )--( -1.802 , 0.868 );
\node at ( -1.247 , 1.564 ){$ 1 $};
\draw[blue! 100 !red] ( -0.445 , 1.95 )--( -1.802 , 0.868 );
\draw[blue! 100 !red] ( 2.0 , 0.0 )--( -1.802 , -0.868 );
\draw[blue! 100 !red] ( 1.247 , 1.564 )--( -1.802 , -0.868 );
\draw[blue! 100 !red] ( -0.445 , 1.95 )--( -1.802 , -0.868 );
\draw[red! 100.0 !blue,dashed] ( -1.802 , 0.868 )--( -1.802 , -0.868 );
\draw[red! 100.0 !blue,dashed] ( 2.0 , 0.0 )--( -0.445 , -1.95 );
\draw[red! 100.0 !blue,dashed] ( 1.247 , 1.564 )--( -0.445 , -1.95 );
\draw[blue! 100.0 !red, line width=0.45mm] ( -0.445 , 1.95 )--( -0.445 , -1.95 );
\draw[blue! 100 !red] ( -1.802 , 0.868 )--( -0.445 , -1.95 );
\draw[blue! 100 !red] ( -1.802 , -0.868 )--( -0.445 , -1.95 );
\draw[blue! 100.0 !red, line width=0.45mm] ( 2.0 , 0.0 )--( 1.247 , -1.564 );
\draw[red! 100.0 !blue,dashed]( 1.247 , 1.564 )--( 1.247 , -1.564 );
\draw[red! 100.0 !blue,dashed] ( -0.445 , 1.95 )--( 1.247 , -1.564 );
\draw[blue! 100 !red] ( -1.802 , 0.868 )--( 1.247 , -1.564 );
\draw[blue! 100 !red] ( -1.802 , -0.868 )--( 1.247 , -1.564 );
\draw[blue! 100.0 !red,line width=0.45mm] ( -0.445 , -1.95 )--( 1.247 , -1.564 );
\filldraw ( 2.0 , 0.0 ) circle [radius=.1];
\filldraw ( 1.247 , 1.564 ) circle [radius=.1];
\filldraw ( -0.445 , 1.95 ) circle [radius=.1];
\filldraw ( -1.802 , 0.868 ) circle [radius=.1];
\filldraw ( -1.802 , -0.868 ) circle [radius=.1];
\filldraw ( -0.445 , -1.95 ) circle [radius=.1];
\filldraw ( 1.247 , -1.564 ) circle [radius=.1];
\end{tikzpicture}
\end{center}
\end{minipage}
\caption{Weighted graphs for the nontrivial facet normals of $\tau_7$}
\label{f7fig}
\end{figure}

We note here the emergence of a facet normal (number 7) with all entries $1 \pmod{3}$.  In terms of triangle decompositions, this constraint has interesting implications which we consider in Section~\ref{sec:3bin}.

\begin{ex}
\label{1mod3}
The facet normal in row $7$ of Table~\ref{f7} has all entries $1 \pmod{3}$; it follows that both alternatives in Proposition~\ref{mod3}(c) are possible.
\end{ex}


The classification of facets of $\tau_8$ is displayed in Table~\ref{f8}.
The total count of 52367 facets also appears in \cite{Deza} as the degree of anti-cuts in the metric polytope $\text{met}_8$.

Here, for the first time, we encounter facets with all entries $2 \pmod{3}$ other than cuts.  We also notice that the trivial facet and balanced cut facets have by far the largest degrees; this is likely due to the relatively large number of triangles orthogonal to the corresponding vectors.

The number of facets of $\tau_n$ for $n=5,6,7,8$ is now sequence \href{https://oeis.org/A246427}{A246427} in the OEIS database; see \cite{OEIS}.

\small
\begin{table}[bthp]
\begin{tabular}{rlrr}
& representative & $\#$ & deg \\
\hline
1. & (1, 0, 0, 0, 0, 0, 0, 0, 0, 0, 0, 0, 0, 0, 0, 0, 0, 0, 0, 0, 0, 0, 0, 0, 0, 0, 0, 0) & 28 & 18848 \\
2. & (1, 1, 0, 1, 0, 0, 0, 0, 0, 0, 0, 0, 0, 0, 0, 0, 0, 0, 0, 0, 0, -1, 0, 0, 0, 1, 1, 1)& 560 & 82 \\
3. & (1, 1, 0, 1, 0, 0, 1, 0, 0, 0, 0, 1, 1, 1, -1, 0, 1, 1, -1, 1, 0, -1, 0, 0, 0, 0, 1, 1)& 3360 & 52 \\
4. & (1, 1, 0, 1, 0, 0, 1, 0, 0, 0, 1, 0, 0, 0, 0, 1, 0, 0, 0, 0, 0, -1, 0, 0, 0, 0, 0, 0)& 56 & 82\\
5. & (1, 1, 0, 1, 0, 0, 1, 0, 0, 0, 1, 0, 0, 0, 0, 0, 1, 1, 1, 1, -1, -1, 0, 0, 0, 0, 0, 1)& 840 & 902\\
6. & (2, 0, 0, 0, 0, 0, 0, 0, 0, 0, 0, 0, 0, 0, 0, 0, 0, 0, 0, 0, 0, -1, -1, 1, 1, 1, 1, 1)& 168 & 1580\\
7. & (2, 1, 1, 0, 0, 1, 0, 0, 1, 0, 0, 0, 1, 0, 0, 0, 0, -1, 0, 0, 0, -1, -1, 0, 1, 1, 1, 1)& 3360 & 125\\
8. & (2, 1, 1, 1, 1, 0, 1, 1, 0, 0, 1, 1, 0, 0, 0, 0, 0, 1, 1, 1, -1, -1, -1, 0, 0, 0, 0, 1)& 3360 & 245\\
9. & (2, 1, 1, 1, 1, 0, 1, 1, 0, 0, 1, 1, 0, 0, 0, -1, -1, 0, 0, 0, 0, -1, -1, 0, 0, 0, 0, 2)& 420 & 27\\
10. & (2, 2, 2, 0, 0, 0, 0, 0, 0, 0, 0, 0, 0, 0, 0, 0, 0, 0, 0, 0, 0, -1, -1, -1, 1, 1, 1, 1)& 280 & 347\\
11. & (2, 2, 2, 2, 2, 2, -1, -1, -1, -1, -1, -1, -1, -1, 2, -1, -1, -1, -1, 2, 2, -1, -1, -1, -1, 2, 2, 2)& 35 & 11878\\
12. & (2, 2, 2, 2, 2, 2, 2, 2, 2, 2, -1, -1, -1, -1, -1, -1, -1, -1, -1, -1, 2, -1, -1, -1, -1, -1, 2, 2)& 56 & 4641\\
13. & (2, 2, 2, 2, 2, 2, 2, 2, 2, 2, 2, 2, 2, 2, 2, -1, -1, -1, -1, -1, -1, -1, -1, -1, -1, -1, -1, 2) & 28 & 245\\
14. & (3, 2, 1, 2, 1, 0, 0, 1, 0, 0, 0, 1, 0, 0, 0, 0, -1, 0, 0, 0, 0, -2, -1, 0, 0, 2, 2, 2)& 10080 & 27\\
15. & (4, 4, -2, 1, 1, 1, 1, 1, 1, 4, 1, 1, 1, -2, -2, -2, 4, -2, 1, 1, 1, -2, -2, 4, 1, 1, 1, 4)& 2016 & 95\\
16. & (4, 4, 4, 4, 4, -2, 1, 1, 1, 1, 1, 1, 1, 1, -2, -2, -2, 4, -2, 1, 1, -2, -2, -2, 4, 1, 1, 4)& 5040 & 60\\
17. & (5, 5, 2, 2, -1, -1, 2, -1, -1, 2, -1, 2, 2, -1, -1, -1, 2, 2, -1, -1, 2, -4, -1, -1, 2, 2, 5, 5)& 2520 & 109\\
18. & (7, 7, 4, 4, 1, 1, 4, 1, 1, -2, 1, 4, -2, 1, 1, 1, -2, 4, 1, 1, -2, -5, -2, -2, 1, 1, 4, 4)& 10080 & 27\\
19. & (8, 5, -1, 5, -1, 2, 2, 2, 5, 5, 2, 2, -1, -1, -4, -4, 2, -1, -1, 2, 2, -4, -4, 5, 5, 2, 2, 8)& 10080 & 27\\
\hline
& total & 52367
\end{tabular}
\caption{Isomorphism classes of facet normals of $\tau_8$, in standard form}
\label{f8}
\end{table}
\normalsize


It is presently out of reach to compute and classify all facets of $\tau_n$ for $n \ge 9$.  However, for $n=9$, we sampled a large number of `random' facets using standard  elimination steps.  We found $143$ isomorphism classes of facets of $\tau_9$, accounting for nearly 12 million distinct facets.   This possibly represents a complete classification, since all but five types have had their neighborhoods exhaustively checked (and are adjacent to no new types).  The trivial and $(5,4)$-cut facets have by far the largest degrees and may be particularly challenging to fully check.  See \cite{DFPS} for more detail on the `adjacency decomposition' method for symmetric cones and polytopes.
A list of the known facets of $\tau_9$ can be found at the second author's webpage:
\url{http://www.math.uvic.ca/~dukes/facets-tri9.txt}

A new feature that emerges at $n=9$ is the existence of automorphism-free facets.

\begin{ex}
With coordinates given in colex order, the facet of $\tau_9$ which is normal to
$$(4, 2, 2, 2, 0, 0, 1, 1, -1, 1, 1, -1, -1, 1, 2, 0, 0, 2, 0, -1, 1, -1, 1, 1, -1, 0, 0, 1, -2, -2, 0, 2, 1, 3, 2, 3)$$
has no automorphisms, and hence generates $9!$ distinct facets of $\tau_9$ under the action of $\mathcal{S}_9$.
\end{ex}

We find it interesting that the number of facets of $\tau_9$ is already so large.  As $n$ grows, if the number of facets of $\tau_n$ exceeds $2^{\binom{n}{2}}$, then it would follow that certain inequalities are only useful to exclude (non-simple) multigraphs from the cone.  A starting estimate on the number of facets of $\tau_n$ via the metric polytope can be obtained from \cite{GYY}.

\section{Lifting facets}

Our aim here concerns lifting facets of $\tau_n$ to facets of $\tau_{n+1}$ via a `vertex splitting' operation.

\begin{prop}
\label{vertex-split}
Let $n \ge 5$ and suppose $y$ is a facet normal of $\tau_n$.  Suppose there exists a triangle $K$ on $[n-1]$ with $\langle y, \mathds{1}_K \rangle >0$.  Define the vector $y^{\text{spl}}$ on $\binom{[n+1]}{2}$ by
$$y^{\text{spl}}(e) = 
\begin{cases}
y(e)  & \text{if $e \subset [n]$}, \\ 
y(\{i,n\}) & \text{if $e=\{i,n+1\}$ for $i \in [n-1]$},\\
-2 \min \{y(\{i,n\}):i \in [n-1] \} & \text{if $e=\{n,n+1\}$}.
\end{cases}$$
Then $y^{\text{spl}}$ is a facet normal of $\tau_{n+1}$.
\end{prop}

\begin{proof}
It is straightforward to check that $y^{\text{spl}}$ is nonnegative on all triangles.  Let $\mathcal{K}$ consist of $K$, together with all triangles $L$ on $[n+1]$ such that $\langle y^{\text{spl}}, \mathds{1}_L \rangle=0$.  Using that $y$ is a facet normal of $\tau_n$ and vertices $n, n+1$ are clones with respect to $y^{\text{spl}}$, every edge in $\binom{[n+1]}{2}$, except possibly $\{n,n+1\}$, is a linear combination of triangles in $\mathcal{K}$.  By the choice of weight on $\{n,n+1\}$, there exists $j$ such that the triangle on $\{j,n,n+1\}$ belongs to $\mathcal{K}$.  Then, since $\{j,n\}$ and $\{j,n+1\}$ are spanned by $\mathcal{K}$, so is $\{n,n+1\}$.
\end{proof}

The third author's dissertation presents a similar construction which allows copies of a facet of $\tau_n$ to be glued together on a common positive triangle to produce facets of $\tau_m$ for $m >n \ge 5$.  This allows for somewhat more general lifts of facets.  See \cite[Proposition 3.8]{Kseniya} for details.

Figure~\ref{small-splits} shows the effect of vertex splitting small facets; refer to Tables~\ref{f6}, \ref{f7}, \ref{f8} for facet labels, which are indicated in the figure as subscripts.

\begin{figure}[htbp]
\tikzstyle{vertex}=[circle, draw]
\begin{tikzpicture}[scale=0.75]
\node[vertex][](f5_1) at (10,10) {$5_1$};

\node[vertex][](f6_1) at (6,8) {$6_1$};
\node[vertex][](f6_2) at (8,8) {$6_2$};
\node[vertex][](f6_3) at (10,8) {$6_3$};
\node[vertex][](f6_4) at (12,8) {$6_4$};

\draw[->] (f5_1)--(f6_3);
\draw[->] (f5_1)--(f6_4);

\node[vertex][](f7_1) at (4,6) {$7_1$};
\node[vertex][](f7_2) at (6,6) {$7_2$};
\node[vertex][](f7_3) at (8,6) {$7_3$};
\node[vertex][](f7_4) at (10,6) {$7_4$};
\node[vertex][](f7_5) at (12,6) {$7_5$};
\node[vertex][](f7_6) at (14,6) {$7_6$};
\node[vertex][](f7_7) at (16,6) {$7_7$};

\draw[->] (f6_1)--(f7_1);
\draw[->] (f6_2)--(f7_3);
\draw[->] (f6_2)--(f7_4);
\draw[->] (f6_3)--(f7_5);
\draw[->] (f6_4)--(f7_5);
\draw[->] (f6_4)--(f7_6);

\node[vertex][](f8_1) at (1,4) {$8_1$};
\node[vertex][](f8_4) at (2.5,4) {$8_4$};
\node[vertex][](f8_5) at (4,4) {$8_5$};
\node[vertex][](f8_6) at (5.5,4) {$8_6$};
\node[vertex][](f8_7) at (7,4) {$8_7$};
\node[vertex][](f8_8) at (8.5,4) {$8_8$};
\node[vertex][](f8_9) at (10,4) {$8_9$};
\node[vertex][](f8_10) at (11.5,4) {$8_{10}$};
\node[vertex][](f8_11) at (13,4) {$8_{11}$};
\node[vertex][](f8_12) at (14.5,4) {$8_{12}$};
\node[vertex][](f8_13) at (16,4) {$8_{13}$};
\node[vertex][](f8_15) at (17.5,4) {$8_{15}$};
\node[vertex][](f8_16) at (19,4) {$8_{16}$};

\draw[->] (f7_1)--(f8_1);
\draw[->] (f7_2)--(f8_5);
\draw[->] (f7_2)--(f8_7);
\draw[->] (f7_2)--(f8_8);
\draw[->] (f7_3)--(f8_4);
\draw[->] (f7_3)--(f8_6);
\draw[->] (f7_4)--(f8_6);
\draw[->] (f7_4)--(f8_9);
\draw[->] (f7_4)--(f8_10);
\draw[->] (f7_5)--(f8_11);
\draw[->] (f7_5)--(f8_12);
\draw[->] (f7_6)--(f8_12);
\draw[->] (f7_6)--(f8_13);
\draw[->] (f7_7)--(f8_15);
\draw[->] (f7_7)--(f8_16);
\end{tikzpicture}
\caption{Vertex splitting facets of $\tau_n$ for $n=5,6,7$}
\label{small-splits}
\end{figure}

\begin{ex}
The weighted graph shown in Figure~\ref{bin-star} defines a facet of $\tau_8$.  Repeatedly applying Proposition~\ref{vertex-split} to pendant vertices gives an infinite family of facets of $\tau_n$, $n \ge 8$, as follows.  Given any partition $(A,B)$ of $\{3,\dots,n\}$ with $|A|,|B| \ge 3$, a facet arises from the normal vector $y$ defined by
\begin{equation}
\label{binary-star}
y(e) = 
\begin{cases}
-1  & \text{if $e =\{1,2\}$}, \\ 
1 & \text{if $e=\{1,a\}$ for $a \in A$, or $e=\{2,b\}$ for $b \in B$},\\
0 & \text{otherwise}.
\end{cases}
\end{equation}

\begin{figure}[htbp]
\begin{center}
\begin{tikzpicture}
\draw[color=red,dashed] (1,2) -- (5,2);
 \node at (3,2.3){$-1 $};
  \node at (0.2,1.2){$1 $};
\foreach \x in {0,1,2}{
\draw[color=blue] (\x,0) -- (1,2);
}
\foreach \x in {4,5,6}{
\draw[color=blue] (\x,0) -- (5,2);
}
\fill (1,2) circle [radius=.1]; \fill (5,2) circle [radius=.1]; 
\foreach \x in {0,1,2,4,5,6}{
\fill (\x,0) circle [radius=.1];
}
\end{tikzpicture}
\caption{The weighted graph corresponding to the `binary star' facet of $\tau_8$}
\label{bin-star}
\end{center}
\end{figure}
\end{ex}

The terminology `binary star' was used for such facets in \cite{Kseniya}.  These have some significance for triangle decompositions.  Consider the question of whether $G$ has a  fractional triangle decomposition under the assumption that it has at least $\frac{3}{4} \binom{n}{2}$ edges and has minimum degree $\delta(G) \ge cn$.  The binary star $y$ reveals that we cannot take $c<\frac{1}{2}$.  For instance, suppose $(A,B)$ is a roughly balanced partition of $\{3,\dots,n\}$.  Build the graph $G$ on vertex set $[n]$ from a clique on $A \cup B$, and join every vertex in $A$ to $2$, every vertex in $B$ to $1$, and finally include the edge $\{1,2\}$.  We have $\langle y,\mathds{1}_G \rangle = -1$ where $y$ is as in \eqref{binary-star}; indeed, the edge $\{1,2\}$ belongs to no triangle in $G$.

\section{A class of nearly-Euclidean facets}
\label{sec:3bin}

Recall the interpretation of cut facets as maps $V(G) \rightarrow \{0,1\}$, where the implied constraint on triangle decomposability of $G$ amounts to the perimeter bound.  In a sample of several hundred facets of $\tau_n$ in the range $8 \le n \le 12$, we notice many that roughly resemble a similar Euclidean embedding to the interval $[0,1]$.   Here, we examine in detail one such class of facets.

Let us start by considering an analog of cuts with three `bins': we let $\omega:[n] \rightarrow \{0,\frac{1}{2},1\}$, and define $y_\omega \in \R^{\binom{n}{2}}$ to have entries
$$y_\omega(\{i,j\}) = 4-6|\omega(i)-\omega(j)| \in \{-2,1,4\}.$$
With $A=\omega^{-1}(0)$, $B=\omega^{-1}(\frac{1}{2})$ and $C=\omega^{-1}(1)$, the corresponding weighted graph is depicted in Figure~\ref{y-omega}.  Note that when $B=\emptyset$, $y_\omega$ reduces to a cut with partition $(A,C)$.

\begin{figure}[htbp]
\begin{center}
\begin{tikzpicture}
\node at (0.5,2.25){$A$};
\draw[rounded corners] (0,0) rectangle (1,2);

\node at (2.5,2.25){$B$};
\draw[rounded corners] (2,0) rectangle (3,2);

\node at (4.5,2.25){$C$};
\draw[rounded corners] (4,0) rectangle (5,2);

\draw[color=blue] (1,1) -- (2,1);
\draw[color=blue] (3,1) -- (4,1);

\draw [color=red,dashed] (1,0.5) to[out=300,in=240] (4,0.5);

\node at (1.5,1.25){$1$};
\node at (3.5,1.25){$1$};
\node at (2.5,-.5){$-2$};

\foreach \x in {0,2,4}
 {
 \draw[color=blue,line width=0.45mm] (\x+0.5,0.5)--(\x+0.5,1.5);
 \node at (\x+0.2,1){$4$};
 }
\end{tikzpicture}
\caption{Weighted graph for the supporting vector $y_\omega$}
\label{y-omega}
\end{center}
\end{figure}

It is easy to see that $y_\omega$ supports $\tau_n$.
However, $y_\omega$ is not a facet normal unless $B = \emptyset$, since otherwise it is the sum of two cuts with partitions $(A \cup B,C)$ and $(A,B\cup C)$.  However, if we slightly modify $y_\omega$ by reducing some values in the middle bin $\binom{B}{2}$ from $4$ to $-2$, then in some cases the resulting vector is a facet normal.  For example, the facet of $\tau_7$ mentioned in Example~\ref{1mod3} has precisely this structure, where $A$ and $C$ are singletons and all edges in a cycle on the five vertices of $B$ have been reduced from $4$ to $-2$.  More generally, one can subtract $6 \mathds{1}_H$ from $y_\omega$, where $H$ is a maximal triangle-free graph with $V(H)=B$ and still have the resulting vector support $\tau_n$.
We note that, although maximal triangle-free, the complete bipartite graphs do not induce facets in this way.

\begin{prop}
The vector $y_\omega-6\mathds{1}_H$ is not a facet normal if $H$ is complete bipartite.
\end{prop}
\begin{proof}
Let $B_1 \cup B_2$ be the bipartition of $H$.  Then $y_\omega-6\mathds{1}_H = y_1 + y_2$, where
$y_1$ is the cut corresponding to $(A \cup B_1,C \cup B_2)$ and $y_2$ is the cut corresponding to $(A \cup B_2,C \cup B_1)$.  It follows that the given vector is not a facet normal.
\end{proof}

From the limited experimentation we have done, it appears the complete bipartite graphs may in fact be the only maximal triangle-free graphs which do not induce facets.  We now describe a class of graphs $H$ which always work.

A $C_5$-\emph{blow-up} is a graph obtained by replacing every vertex of a $5$-cycle by an independent set (of possibly different sizes).  
Such graphs are clearly maximal triangle-free; moreover, they can be constructed recursively by starting with a `seed' $C_5$, and adding vertices one at a time, joining them to maximal independent sets.

\begin{prop}
\label{c5-facets}
Let $\omega:[n] \rightarrow \{0,\frac{1}{2},1\}$ with nonempty level sets $A,B,C$.
Let $H$ be a $C_5$-blow-up spanning the vertex set $B=\omega^{-1}(\frac{1}{2})$.  Then $y_\omega-6\mathds{1}_H$ is a facet normal of $\tau_n$.
\end{prop}

\begin{proof}
By Proposition~\ref{vertex-split}, it suffices to consider the case $|A|=|C|=1$, so that $k:=|B|=n-2$.  Assume $k \ge 5$.
When $k=5$, the only candidate graph for $H$ is $C_5$, resulting in the facet normal of Example~\ref{1mod3}.
Suppose for induction hypothesis that the result holds for $n$ vertices, where $n \ge 7$.  Consider $\omega$ with level sets $A,B,C$, where $|A|=|C|=1$ and with a $C_5$-blow-up $H$ on the $k+1$ vertices of $B$.  This $H$ can be obtained by cloning a vertex, wlog $k$, in a $C_5$-blow-up $H'$ of order $k$ on vertex set $B'$.  Using the induction hypothesis, let $z$ be the normal to the $(1,k,1)$-facet based on $H'$.  Choose $e \in \binom{B' \setminus \{k\}}{2}$ which is not an edge of $H'$.  Then $e \cup A$ defines a triangle $K$ such that $\langle z,\mathds{1}_K \rangle = 1+1+4 > 0$.  By 
Proposition~\ref{vertex-split}, the vector $z^{\text{spl}}=y_\omega - 6 \mathds{1}_H$ obtained by splitting at $k$ produces a facet normal of $\tau_{n+1}$.  
\end{proof}

Let us call facets of the type described in Proposition~\ref{c5-facets} $C_5$-\emph{facets}.
The non-isomorphic $C_5$-blow-ups on $k$ vertices admit a natural bijection with the bracelets made from $5$ black beads and $k-5$ white beads.  The counting problem for such bracelets appears in the OEIS database as sequence \href{https://oeis.org/A032279}{A032279}, see ~\cite{OEIS}.  Let $a_k$ denote the $k$th term of this sequence.  By a formula of Robert Israel, $a_k = k^4/240+o(k^3)$.  It follows that there are
$$\sum_{k=5}^{n-2}\left\lfloor\frac{n-k}{2}\right\rfloor a_k = \Omega(n^6)$$ non-isomorphic $C_5$-facets of $\tau_n$.

To display the utility of these facets, we construct graphs of minimum degree near $3n/4$ which are excluded from $\tau_n$ by a $C_5$-facet but no cut facets.  The construction is a minor perturbation of the canonical example $C_4 \cdot K_{n/4}$ rejected by cut facets.

\begin{cons}
Fix $\epsilon>0$, and let $J_p$ denote a graph on $p$ vertices of degree $(1-\epsilon)p$.
Let $n=4p+5q$ and construct $G$ as the join $C_4 \cdot J_p + C_5 \cdot \overline{K_q}$ on vertex set $[n]$.

We build a $C_5$-facet based on $H=C_5 \cdot \overline{K_q}$ which witnesses that $G$ is not in $\tau_n$.
Assign the copies of $J_p$ `alternately' into $A$ and $C$, and let $H$ be the subgraph $C_5 \cdot \overline{K_p}$ on vertex set $B$.  Then, we compute
\begin{equation}
\label{c5-test}
\langle y_\omega - 6 \mathds{1}_H, \mathds{1}_G \rangle = 4 \times 2(1-\epsilon)p^2 +1 \times (4p)(5q)-2 \times (4p^2+ 5q^2) = -8\epsilon p^2+20pq-10q^2.
\end{equation}
If, instead, we attempt to place vertices of $H$ into $A \cup C$, we get at least $p^2$ edges of $H$ internal to $A$ or $C$.  It follows that such a cut $z$ satisfies
\begin{equation}
\label{cut-test}
\langle z,\mathds{1}_G \rangle \ge 8(1-\epsilon)p^2 + 20pq -4q^2.
\end{equation}
Let $p$ be large and take $q$ to be an integer near $(\frac{2}{5} \epsilon + \frac{8}{125} \epsilon^2) p$.  With this choice, the right side of \eqref{c5-test} is negative while the right side of \eqref{cut-test} is positive.  

The minimum degree of $G$ is $\delta(G)=(3-\epsilon)p+5q$, or roughly $\frac{75+25\epsilon+8\epsilon^2}{100+50\epsilon+8\epsilon^2}n$, which approaches (from below) the threshold in Conjecture~\ref{nw-conj}.  
\end{cons}

\section{Concluding Remarks}

We have seen many properties of the cone $\tau_n$, and connected its halfspace description with the triangle decomposition problem for graphs.  
We have pointed out several places where additional work could lead to a better understanding of facets. 
As another next step, we feel it would be useful to attempt an approximation of $\tau_n$ for large $n$ by a cone with simpler structure.  

The interested reader may wish to explore some related directions we have neglected to mention so far.
First, the cone $\tau_n$ even long ago attracted some interest in quantum physics for its connection with the `$N$-representability problem'; see for instance \cite{McRae}.  Next, it is worth mentioning the `cut cone' in $\R^{\binom{n}{2}}$, generated by all vertex cuts in the complete graph $K_n$. More information can be found in \cite{cut-cone-1,cut-cone-3}.  In particular, the `uniform cut cone' studied by Neto \cite{Neto} is closely related to $\tau_n$.

Finally, our cone $\tau_n$ is a special case of the family of cones introduced in \cite{DW}.  This more general setting considers the cone generated by the inclusion matrix of $t$-subsets $\binom{[n]}{t}$ versus $k$-subsets $\binom{[n]}{k}$, where $t,k,n$ are positive integers satisfying $k \ge t$ and $n \ge k+t$.  Alternatively, this is the cone of weighted $t$-uniform hypergraphs on $n$ vertices generated by $k$-vertex cliques $K_k^{(t)}$.  It was shown in \cite{DW} that various inequalities for $t$-designs arise from certain supporting vectors described by orthogonal polynomials.

\section*{Acknowledgements}

We are grateful to Richard M. Wilson, who introduced the second author to $\tau_n$, to Michel and Antoine Deza for information on the metric polytope leading to Proposition~\ref{met-n}, to  Haggai Liu, who supplied fast Python code to find lexicographically largest representatives for our library of isomorphism types of facet normals for $n \le 9$, and to the referees for careful reading which improved the manuscript from its initial state.

\end{document}